\documentclass[a4paper, 12pt]{amsart}
\usepackage{amsmath,amssymb,amsthm}
\usepackage{amscd}
\usepackage{array,enumerate}
\newtheorem{Thm}{Theorem}[section]

\newtheorem{Lem}[Thm]{Lemma}

\newtheorem{Thmint}{Theorem}[section]
\newtheorem{Propint}[Thmint]{Proposition}

\newtheorem{Corint}[Thmint]{Corollary}

\theoremstyle{definition}
\newtheorem{Rem}[Thm]{Remark}

\newcommand{\Cs}{C$^\ast$}

\newcommand{\sd}{^{\ast\ast}}
\newcommand{\id}{\mbox{\rm id}}

\newcommand{\rg}{\mathop{{\mathrm C}_{\mathrm r}^\ast}}

\newcommand{\rc}{\mathop{\rtimes _{\mathrm r}}}

\newcommand{\vnc}{\mathop{\bar{\rtimes}}}

\newcommand{\vrca}[1]{\mathop{\bar{\rtimes}_{#1}}}
\newcommand{\rca}[1]{\mathop{\rtimes _{{\mathrm r}, #1}}}

\DeclareMathOperator{\SL}{SL}

\DeclareMathOperator{\bigfp}{\lower0.25ex\hbox{\LARGE $\ast$}}

\title[Fixed point \Cs-algebras]
{On pathological properties of fixed point algebras in Kirchberg algebras}
\author{Yuhei Suzuki}
\subjclass[2000]{Primary~
22D25, 46L55, Secondary~46L05}
\keywords{Fixed point algebras, nuclearity, non-commutative dynamical systems}
\address{Graduate school of mathematics, Nagoya University, Chikusaku, Nagoya, 464-8602, Japan}
\email{yuhei.suzuki@math.nagoya-u.ac.jp}
\begin{document}

\begin{abstract}
We investigate how the fixed point algebra of a \Cs-dynamical system can differ from the underlying \Cs-algebra.
For any exact group $\Gamma$ and any
 infinite group $\Lambda$, we construct
an outer action of $\Lambda$ on the Cuntz algebra $\mathcal{O}_2$
whose fixed point algebra is almost equal to the reduced
group \Cs-algebra $\rg(\Gamma)$.
Moreover, we show that every infinite group admits outer actions on all
Kirchberg algebras whose fixed point algebras fail the completely bounded approximation property.
\end{abstract}
\maketitle
\section{Introduction}
In the celebrated paper \cite{Con},
Connes observed in Section 6 that injectivity of von Neumann algebras passes to the fixed point algebras
of amenable group actions.
In contrast to this observation, in the seminal paper \cite{Kir94},
Kirchberg showed that any unital separable exact \Cs-algebra
can be realized as a (liftable) quotient
of the fixed point algebra of an (inner) automorphism on a UHF-algebra.
This in particular implies that nuclearity need \emph{not} pass to the fixed point algebra of an amenable group action.
While Kirchberg's theorem indicates bad behavior of the operation
taking the fixed point algebra,
it is also true that the fixed point algebras play important roles to understand
\Cs-dynamical systems and associated \Cs-algebras and invariants in some situations; see e.g., \cite{Kas}, \cite{Sz18}.

Motivated by these results, we are interested in knowing
how bad the fixed point algebras of general amenable groups
of a nuclear \Cs-algebra can be.
We particularly examine this on one of the most ubiquitous (nuclear) \Cs-algebras---the Cuntz algebra $\mathcal{O}_2$ \cite{Cun}.
\begin{Thmint}\label{Thm}
Let $\Gamma$ be a countable exact group.
Let $\Lambda$ be an infinite countable group.
Then $\Lambda$ admits an outer action on $\mathcal{O}_2$
whose fixed point algebra is isomorphic
to an intermediate \Cs-algebra of
$\rg(\Gamma) \subset L(\Gamma)$.
Moreover, when $\Gamma$ has the approximation property $($AP$)$ \cite{HK},
one can arrange the fixed point algebra to be isomorphic to $\rg(\Gamma)$.
\end{Thmint}
We point out that the infiniteness of $\mathcal{O}_2$ is essential in the statement.
Indeed non-amenable reduced group \Cs-algebras do not embed
into stably finite nuclear \Cs-algebras by Proposition 6.3.2 of \cite{BO}, \cite{Cun78}, and \cite{Haa}.
Note also that many exact groups have the AP.
The class of groups with the AP contains all weakly amenable groups (thus all hyperbolic groups \cite{Ozh}),
and is stable under extensions and free products.
See Section 12.4 of \cite{BO} for details.
It would be interesting to compare Theorem \ref{Thm} with
the following useful statement: The fixed point algebra of a compact group action
is nuclear if and only if the original \Cs-algebra is nuclear. See Section 4.5 of \cite{BO}
for details.

Applying Theorem \ref{Thm} to a locally finite group,
(after a slight refinement of the proof,) we
also obtain refined versions of Theorem A and Corollary B in \cite{Suz17}.
We recall that Watatani \cite{Wat} shows that most good properties
(including nuclearity) are stable under
finite Watatani index inclusions. See Proposition 2.7.2 of \cite{Wat} and the remark below it for the precise statement.
The following corollary shows that this familiar statement does
not extend to ``approximately finite'' index subalgebras.
\begin{Corint}\label{Cor}
There is a descending sequence of irreducible finite Watatani index inclusions
\[\cdots \subset A_{n+1}\subset A_n \subset \cdots \subset A_2 \subset A_1\]
of isomorphs of $\mathcal{O}_2$
whose intersection does not have the operator approximation property nor the local lifting property.
Moreover one can arrange the sequence to satisfy the following homogeneity condition:
for every $m\in \mathbb{N}$, the inclusions $A_{n+m} \subset \cdots \subset A_n$, $n\in \mathbb{N}$, are pairwise isomorphic.
\end{Corint}
Here we say that two sequences of inclusions $A_1 \subset A_2 \subset \cdots \subset A_n$
 and $B_1 \subset B_2 \subset \cdots \subset B_n$ of \Cs-algebras
are \emph{isomorphic}
if there is an isomorphism $\varphi \colon A_n \rightarrow B_n$
satisfying $\varphi(A_i)=B_i$ for $i=1, \ldots, n-1$.
We recall that an inclusion $A\subset B$ of unital simple \Cs-algebras
is said to be \emph{irreducible} if the relative commutant
is trivial: $A'\cap B=\mathbb{C}$.

We also give the following result on general Kirchberg algebras.
\begin{Propint}\label{Prop}
Let $\Lambda$ be a countable infinite group.
Let $A$ be a Kirchberg algebra.
Then $\Lambda$ admits an outer action on $A$
whose fixed point algebra does not have the completely bounded approximation property
 nor the local lifting property.
\end{Propint}
Here we recall that a simple separable nuclear purely infinite \Cs-algebra
is called a \emph{Kirchberg algebra}.
Kirchberg algebras form an important class of \Cs-algebras.
A side of rich structures of Kirchberg algebras is reflected to the successful classification theorem of Kirchberg--Phillips \cite{Kir94b}, \cite{Phi}.
We refer the reader to the book \cite{Ror} for fundamental facts and backgrounds on this subject.

The key ingredients of our constructions (besides well-known deep results)
are ``amenable'' actions of non-amenable groups on Kirchberg algebras
recently obtained in \cite{Suzeq}, \cite{Suz19}.

We expect that our results
give a new intuitive picture of the inaccessibility of
conjugacy classes of \Cs-dynamical systems,
in contrast to successful classification results
\emph{up to cocycle conjugacy} (see e.g., \cite{Nak}, \cite{IM}, \cite{IM2}, \cite{Sz18}, and references therein).

\subsection*{Notations}
Here we fix a few notations used in this paper.
\begin{itemize}
\item For $\epsilon>0$ and for two elements $x$, $y$ of a \Cs-algebra,
denote by $x\approx_{\epsilon} y$ if $\|x -y\| <\epsilon$.
\item For a group action $\alpha \colon \Gamma \curvearrowright A$ on a \Cs-algebra $A$,
denote by $A^\alpha$ the fixed point algebra of $\alpha$:
\[A^\alpha:=\{ a\in A: \alpha_s(a)=a {\rm~for~all~}s\in \Gamma\}.\]
\item The symbols `$\otimes$', `$\rc$', `$\vnc$' stand
for the minimal tensor products, the reduced \Cs-crossed products,
and the von Neumann algebra crossed products respectively.

\end{itemize}
For basic facts on \Cs-algebras and discrete groups, we
refer the reader to the book \cite{BO}.
\section{Proofs, constructions, and remarks}

The following lemma would be well-known for specialists.
For the reader's convenience, we include a proof.
\begin{Lem}\label{Lem:fp}
Let $\Lambda$ be a group.
Let $I$ be a $\Lambda$-set
such that all $\Lambda$-orbits are infinite.
Let $A$ be a unital \Cs-algebra.
Let $\sigma \colon \Lambda \curvearrowright \bigotimes_I A$
be the tensor shift action.
Then 
$(\bigotimes_I A)^{\sigma}=\mathbb{C}$.
\end{Lem}
\begin{proof}
Take $x\in (\bigotimes_I A)^{\sigma}$.
Let $\epsilon >0$.
Fix a state $\varphi$ on $A$.
For each $S\subset I$,
define a conditional expectation $E_S \colon \bigotimes _I A \rightarrow \bigotimes _S A$
to be $E_S:= (\bigotimes_S \id_A) \otimes (\bigotimes_{I \setminus S} \varphi)$.
Choose a finite subset $F$ of $I$ and an element
$y\in \bigotimes _F A$ satisfying
$x \approx_\epsilon y$.
By the assumption on $I$,
one can choose $s\in \Lambda$ satisfying
$sF \cap F =\emptyset$.
(The existence of such $s$ is obvious in applications in the present paper.
For completeness, we give a proof for general case.
To lead to a contradiction, assume that there is no such $s$.
Then one can find sequences $s_1, \ldots, s_n\in \Gamma$ and $x_1, \ldots, x_n \in F$
satisfying $\Gamma = \bigcup_{k=1}^n s_k{\rm Stab}(x_k)$.
Here ${\rm Stab}(x_k)$ denotes the stabilizer subgroup of $x_k$.
Since each ${\rm Stab}(x_k)$ has infinite index in $\Gamma$,
this is a contradiction; for the proof, see \cite{Neu}, Lemma 4.1.)
Observe that
\[\sigma_s(y)\approx_{\epsilon} \sigma_s(x)=x \approx_{\epsilon} y.\]
This implies
\[x \approx_\epsilon y= E_{F}(y) \approx_{2\epsilon} E_{F}(\sigma_s(y))\in \mathbb{C}.\]
Since $\epsilon>0$ is arbitrary, this yields $x\in \mathbb{C}$.
\end{proof}
Recall that an automorphism $\alpha$ of a \Cs-algebra $A$ is said to be \emph{inner}
if there is a unitary multiplier $u$ of $A$
satisfying $\alpha(a)=uau^\ast$ for all $a\in A$.
An action $\alpha$ of a discrete group $\Gamma$ on $A$ is said to be \emph{outer}
if $\alpha_s$ is not inner for all $s\in \Gamma \setminus \{e\}$.
For simple \Cs-algebras, outerness of an action can be regarded as a non-commutative analogue of
(topological) freeness of topological dynamical systems; see \cite{Kis} for instance.

To confirm outerness of actions, we use the central sequence algebras.
Here we briefly recall them.
For a unital \Cs-algebra $A$,
denote by $\ell^\infty(\mathbb{N}, A)$ the \Cs-algebra of all bounded sequences of $A$.
Let $c_0(\mathbb{N}, A)$ denote the ideal of $\ell^\infty(\mathbb{N}, A)$ consisting
of all sequences tending to $0$ in norm.
Set $A^\infty:= \ell^\infty(\mathbb{N}, A)/ c_0(\mathbb{N}, A)$.
Then we have an embedding
$A \rightarrow A^\infty$; $a \mapsto (a, a, \ldots) + c_0(\mathbb{N}, A)$.
By this embedding, we regard $A$ as a \Cs-subalgebra of $A^\infty$.
Define $A_\infty := A^\infty \cap A'$.
This is the central sequence algebra of $A$.
Any automorphism $\alpha$ of $A$
induces an automorphism on $\ell^\infty(\mathbb{N}, A)$
by pointwise application.
This further induces the automorphism $\alpha_\infty$ on $A_\infty$.
Observe that if $\alpha$ is inner, then $\alpha_\infty$
is trivial.
\begin{proof}[Proof of Theorem \ref{Thm}]
By the proof of Proposition B and Remark 3.3 in \cite{Suzeq},
one can take an action $\alpha \colon \Gamma \curvearrowright \mathcal{O}_2$
satisfying the following conditions.
\begin{itemize}
\item $\mathcal{O}_2^\alpha\neq \mathbb{C}$.

\item
Denote by
\[\beta\colon \Gamma \curvearrowright \bigotimes_\Lambda \mathcal{O}_2\]
the diagonal action of copies of $\alpha$ indexed by $\Lambda$.
Then \[A:=\left(\bigotimes_\Lambda \mathcal{O}_2\right) \rca{\beta} \Gamma \cong \mathcal{O}_2.\]
\end{itemize}

Let $\sigma \colon \Lambda \curvearrowright \bigotimes_\Lambda \mathcal{O}_2$
denote the left shift action.
Then $\sigma$ commutes with $\beta$.
Hence $\sigma$ extends to an action
$\theta\colon \Lambda \curvearrowright A$
via the formula
\[\theta_t(x u_s):=\sigma_t(x)u_s {\rm~for~} x\in \bigotimes_\Lambda \mathcal{O}_2, s\in \Gamma, t\in \Lambda.\]
We next show that $\theta_\infty \colon \Lambda \curvearrowright A_\infty$ is faithful, which yields the outerness of $\theta$.
Choose an element $a\in \mathcal{O}_2^\alpha\setminus \mathbb{C}$.
For each $t\in \Lambda$, denote by $a[t]$
the image of $a$ under the embedding of $\mathcal{O}_2$
into the $t$-th tensor product component of $\bigotimes_\Lambda \mathcal{O}_2$.
Observe that each $a[t]$ commutes with $u_s \in A$ for all $s\in \Gamma$.
Fix a state $\varphi$ on $\mathcal{O}_2$.
Then for any $t, u \in \Lambda$ with $t\neq u$,
we have
$\|a[t]-a[u]\|\geq\| [(\bigotimes_{\{t\}} \id_A )\otimes (\bigotimes_{\Gamma \setminus \{t\}} \varphi)](a[t]-a[u])\|= \|\varphi(a)-a\|>0$.
Take a sequence $(t_n)_{n=1}^\infty$ in $\Lambda$ which tends to infinity.
Then the sequence $(a[t_n])_{n=1}^\infty$ defines an
element $x$ of $A_\infty$.
The above inequality shows that 
 $\theta_{\infty, t}(x) \neq x$ for all $t\in \Lambda \setminus \{e\}$.
Thus $\theta_\infty$ is faithful.

Put $M:= (\bigotimes_\Lambda \mathcal{O}_2)\sd$.
Let $\bar{\beta}$ and $\bar{\sigma}$ denote the weakly continuos extensions of $\beta$ and $\sigma$ to $M$ respectively.
Let $\bar{\theta}$ denote the action
of $\Lambda$ on $M\vrca{\bar{\beta}} \Gamma$ defined analogously to $\theta$.
Then the inclusion $\bigotimes_\Lambda \mathcal{O}_2 \subset M$
extends to the $\Lambda$-equivariant inclusion
$A \subset M \vrca{\bar{\beta}} \Gamma$.
Since $(\bigotimes_\Lambda \mathcal{O}_2)^\sigma =\mathbb{C}$ (by Lemma \ref{Lem:fp}),
by Corollary 3.4 of \cite{Suz19}, we have
$A^{\theta} \subset L(\Gamma)$.
Moreover, when $\Gamma$ has the AP, thanks to Proposition 3.4 of \cite{Suz17},
we further obtain
$A^{\theta}=\rg(\Gamma)$.
\end{proof}

\begin{proof}[Proof of Corollary \ref{Cor}]
We fix a non-trivial finite group $G$ and set $\Lambda := \bigoplus_{\mathbb{N}} G$. 
Put $\Gamma:=\SL(3, \mathbb{Z})$. Note that $\Gamma$ is exact
(see Section 5.4 of \cite{BO}) and does not have the AP \cite{LS}.
For each $n\in \mathbb{N}$, set $\Lambda_n :=\bigoplus_{k=1}^n G \subset \Lambda$.
 Then, note that $\Lambda_n \subset \Lambda_{n+1}$ and that $\Lambda = \bigcup_{n=1}^\infty \Lambda_n$.
Take an action $\alpha \colon \Gamma \curvearrowright \mathcal{O}_2$
as in the proof of Theorem \ref{Thm}.
By the construction in \cite{Suzeq} (see Proposition B and Remark 3.3), we may further assume that the fixed point algebra $\mathcal{O}_2^{\alpha}$
admits
a unital embedding $\iota\colon \mathcal{O}_2 \rightarrow \mathcal{O}_2^{\alpha}$.
Define $I:= \Lambda \sqcup \bigsqcup_{n=1}^\infty \Lambda/\Lambda_n$.
We equip $I$ with the left translation $\Lambda$-action.
Let $\beta \colon \Gamma \curvearrowright \bigotimes_I \mathcal{O}_2$
denote the diagonal action of copies of $\alpha$ indexed by $I$.
Let $\sigma \colon \Lambda \curvearrowright \bigotimes_I \mathcal{O}_2$
denote the tensor shift action.
Then $\sigma$ commutes with $\beta$.
Thus $\sigma$ induces the $\Lambda$-action $\theta$
on $A:=(\bigotimes_I \mathcal{O}_2) \rca{\beta} \Gamma$ via the formula
\[\theta_t(x u_s):=\sigma_t(x)u_s {\rm~for~} x\in \bigotimes_I \mathcal{O}_2, s\in \Gamma, t\in \Lambda.\]
For $n\in \mathbb{N}$, set $\theta_n :=\theta|_{\Lambda_n}$
and define $A_n :=A^{\theta_n}$.
Then note that $A_{n+1} \subset A_n$ for all $n\in \mathbb{N}$
and that $\bigcap_{n=1}^\infty A_n= A^\theta$.
As in the proof of Theorem \ref{Thm}, we obtain $\rg(\Gamma) \subset A^\theta \subset L(\Gamma)$.
For $n, m \in\mathbb{N}$, set $\Upsilon_{n, m}:= \bigoplus_{k=n+1}^{n+m} G \subset \Lambda$.
Observe that, as $\Upsilon_{n, m}$ commutes with $\Lambda_n$,
the \Cs-subalgebra $A_n$ is invariant under $\theta(\Upsilon_{n, m})$ for all $m\in \mathbb{N}$.
Let $\vartheta_{n, m} \colon \Upsilon_{n, m} \curvearrowright A_n$ denote the restricted action of $\theta$.
Then, for each $n, m\in \mathbb{N}$,
since $\Lambda_n \cdot \Upsilon_{n, m} = \Lambda_{n+m}$,
we have $A_n^{\vartheta_{n, m}}=A_{n+m}$.
By a similar argument to the proof of Theorem \ref{Thm},
one can show the outerness of $\vartheta_{n, m}$.
By the proof of Corollary B of \cite{Suz17},
the intersection algebra
$\bigcap_{n=1}^\infty A_n$ does not have the operator approximation
property nor the local lifting property.
Since $\theta$ is outer and each $\Lambda_n$ is finite,
the corresponding fixed point algebras $A_n$ are simple and nuclear.
By Remark 3.14 of \cite{Izu}, each inclusion $A_{n} \subset A_1$ is irreducible.
By Corollary 3.12 of \cite{Izu}, each inclusion $A_{n+1} \subset A_n$
has finite Watatani index. (In fact the index is $\sharp G$.)

We next show that for each $n\in \mathbb{N}$, $A_n$ is isomorphic to $\mathcal{O}_2$.
To see this, for each $k\in \mathbb{N}$, let $\iota_{k}$ denote the composite of
the unital embedding $\iota\colon \mathcal{O}_2 \rightarrow \mathcal{O}_2^{\alpha}$
and the canonical embedding
of $\mathcal{O}_2^{\alpha}$ into the $\Lambda_k$-th
tensor product component of $\bigotimes_I \mathcal{O}_2^{\alpha} \subset A$.
Then the sequence $(\iota_{k+n})_{k=1}^\infty$
defines a unital embedding of $\mathcal{O}_2$ into $(A_n)_\infty$.
By Kirchberg's theorem (see \cite{KP}, Lemma 3.7),
$A_n$ is isomorphic to $\mathcal{O}_2$.

Finally we show the homogeneity condition described in the statement.
Let $m\in \mathbb{N}$ be given.
Then, for each $n\in \mathbb{N}$, the sequence $(\iota_{k+n+m})_{k=1}^\infty$ witnesses that $\vartheta_{n, m}$ satisfies condition (iii) of Theorem 4.2 of \cite{Izu04}.
Thus, by Theorem 4.2 of \cite{Izu04},
the actions $\vartheta_{n, m}$, $n\in \mathbb{N}$, are pairwise conjugate
(after the canonical identifications $\Lambda_{n, m}\cong \Lambda_{l, m}$ by shifting indices).
Consequently, the inclusions $A_{n+m} \subset \cdots \subset A_n$, $n\in \mathbb{N}$,
are pairwise isomorphic.
\end{proof}
\begin{Rem}\label{Rem:BE}
The construction in the proof of the Corollary works for \emph{any} countable exact group $\Gamma$ instead of $\SL(3, \mathbb{Z})$.
The isomorphism classes of the resulting inclusions
\[A_{n+m} \subset \cdots \subset A_n\]
do \emph{not} depend on the choice of $\Gamma$ (when $G$ is fixed).
The resulting intersection algebra is an intermediate \Cs-algebra
of $\rg(\Gamma) \subset L(\Gamma)$,
and when $\Gamma$ satisfies the AP,
it is in fact equal to the reduced group \Cs-algebra $\rg(\Gamma)$.
\end{Rem}
\begin{Rem}
One can also arrange the fixed point algebra in
Theorem \ref{Thm} and Remark \ref{Rem:BE} to be $A \rca{\zeta} \Gamma$
when $\Gamma$ has the AP, $A$ is a unital separable nuclear \Cs-algebra,
and $\zeta\colon \Gamma \curvearrowright A$ is an action without $\Gamma$-invariant proper ideals.
Here we sketch how to modify the proof in the former case.
The latter case is similarly obtained.
From now on, we use the notations in the proof of Theorem \ref{Thm}.
We first replace $\alpha$ by the diagonal action of $\alpha$, the left shift action on $\bigotimes_{\Gamma} \mathcal{O}_2$, and the trivial action
on $\mathcal{O}_2$ (the underlying \Cs-algebra is again identified with $\mathcal{O}_2$ by Theorem 3.8 of \cite{KP}).
Denote by $\tilde{\beta}$
the diagonal action of $\beta$ and $\zeta$.
Then the reduced crossed product $B$ of $\tilde{\beta}$ is nuclear by the choice of $\alpha$ (cf.~ \cite{Suzeq}, Proposition B).
Theorem 7.2 of \cite{OP} shows that
 $B$ is simple.
By the choice of $\alpha$, $B \cong B \otimes \mathcal{O}_2$.
By Theorem 3.8 of \cite{KP}, $B\cong \mathcal{O}_2$. 
Set $\tilde{\sigma}:=\sigma \otimes \id_A$.
Define $\tilde{\theta} \colon \Lambda \curvearrowright B$ analogous to $\theta$ by using $\tilde{\sigma}$ instead of $\sigma$.
The proof of Lemma \ref{Lem:fp} shows
that $(\bigotimes_\Lambda \mathcal{O}_2 \otimes A)^{\tilde{\sigma}}=A^\zeta$.  Proposition 3.4 in \cite{Suz17} then yields that $B^{\tilde{\theta}} = A \rca{\zeta} \Gamma$.
\end{Rem}
We finally prove Proposition \ref{Prop}.
The proof involves the (K)K-theory.
For basic facts and terminologies on this subject, we refer the reader to the book \cite{Bla}.
\begin{proof}[Proof of Proposition \ref{Prop}]
Thanks to Kirchberg's $\mathcal{O}_\infty$-absorption theorem (\cite{KP}, Theorem 3.15)
we only need to show the statement for the Cuntz algebra $\mathcal{O}_\infty$.

Let $\Gamma$ be a countable free group of infinite rank.
In the proof of Theorem 5.1 of \cite{Suz19},
we have constructed an action $\alpha \colon \Gamma \curvearrowright \mathcal{O}_\infty$ which contains a unital amenable $\Gamma$-\Cs-subalgebra in the sense of Definition 4.3.1 in \cite{BO}.
By replacing $\alpha$ by the diagonal action of $\alpha$ and the trivial $\Gamma$-action on $\mathcal{O}_\infty$ (cf.~Theorem 3.15 of \cite{KP}) if necessary,
we may assume that $\mathcal{O}_\infty^{\alpha} \neq \mathbb{C}$.
Denote by $\beta$ the diagonal action of copies of $\alpha$ indexed by $\Lambda$. 

By Theorem 8.4.1 (iv) in \cite{Ror}, one can
choose an action $\gamma \colon \Gamma \curvearrowright C$
on a unital Kirchberg algebra $C$ satisfying the universal coefficient theorem \cite{RS} with the following conditions:
\begin{enumerate}
\item $(K_0(C), [1]_0, K_1(C))\cong (\bigoplus_\Gamma \mathbb{Z},  0, 0)$,
\item the induced action $\Gamma \curvearrowright K_0(C)$ is conjugate
to the left shift action $\Gamma \curvearrowright \bigoplus_\Gamma \mathbb{Z}$.
\end{enumerate}
Observe that, thanks to \cite{FKK} (see also \cite{KOS}),
by composing $\gamma_s$ with an approximately inner automorphism
for each canonical generator $s$ of $\Gamma$ if necessary,
we may assume that $\gamma$ admits an invariant (pure) state.
(This manipulation does not affect the above conditions
since approximately inner automorphisms act trivially on the K-groups.)
By the Pimsner--Voiculescu exact sequence \cite{PV},
conditions (1) and (2) imply the isomorphism
\[(K_0(C\rca{\gamma} \Gamma), [1]_0, K_1(C\rca{\gamma} \Gamma))\cong (\mathbb{Z}, 0, 0).\]
(See the second paragraph of the proof of Theorem 5.1 in \cite{Suz19} for details of computation.)

Let $\zeta \colon \Gamma \curvearrowright \bigotimes_\Gamma \mathcal{O}_\infty$
denote the tensor shift action.
 Let $\theta$ denote the diagonal action of $\gamma$ and $\zeta$,
and let $\eta$ denote the diagonal action of $\theta$ and $\beta$.
By the Pimsner--Voiculescu exact sequence \cite{PV},
the inclusion map
\[C \rca{\gamma} \Gamma \rightarrow D:=\left(C \otimes \bigotimes_{\Gamma}\mathcal{O}_\infty\otimes \bigotimes_{\Lambda}\mathcal{O}_\infty \right) \rca{\eta} \Gamma\] induces an isomorphism on $K$-theory.
Note also that $D$
satisfies the universal coefficient theorem by \cite{PV} (see also Corollary 7.2 in \cite{RS}).
By the choice of $\alpha$, $D$ is nuclear (cf.~\cite{Suzeq}, Proposition B).
By Kishimoto's theorem \cite{Kis},
one can conclude that $C \rca{\gamma} \Gamma$, $(C\otimes \bigotimes_{\Gamma} \mathcal{O}_\infty) \rca{\theta} \Gamma$, and $D$ are simple and purely infinite
(see e.g.~Lemma 6.3 of \cite{Suz19b} for details).
By Cuntz's theorem \cite{CunK}, one can take a projection $p$ in $C \rca{\gamma}  \Gamma$
which represents a generator of $K_0(C \rca{\gamma} \Gamma)$.
By the classification theorem of Kirchberg--Phillips \cite{Kir94b}, \cite{Phi}, we obtain
$p D p \cong \mathcal{O}_\infty$.

Next let $\xi\colon \Lambda \curvearrowright C \otimes \bigotimes_{\Gamma}\mathcal{O}_\infty\otimes \bigotimes_{\Lambda}\mathcal{O}_\infty$ denote the diagonal action of the trivial action $\Lambda \curvearrowright C \otimes \bigotimes_{\Gamma}\mathcal{O}_\infty$ and the tensor shift action $\Lambda \curvearrowright \bigotimes_\Lambda \mathcal{O}_\infty$.
Then $\xi$
commutes with $\eta$.
Therefore $\xi$ extends to the action \[\sigma \colon \Lambda \curvearrowright D\]
satisfying $\sigma_t(u_s)= u_s$ for all $s\in \Gamma$ and $t\in \Lambda$.
Observe that $p$ is $\sigma$-invariant.
Therefore $\sigma$ restricts to the action
\[\varsigma\colon \Lambda \curvearrowright pD p.\]
We show that $\varsigma$ satisfies the desired properties.
By the same argument as in the proof of Theorem \ref{Thm},
one can check that $\varsigma$ is outer.
Observe that
\[(pDp)^\varsigma = p D^\sigma p.\]
The proof of Lemma \ref{Lem:fp} together with Proposition 3.4 of \cite{Suz17} shows that
\[D^\sigma = \left(C \otimes \bigotimes_{\Gamma}\mathcal{O}_\infty\right) \rca{\theta} \Gamma.\]
Since $C$ admits a $\gamma$-invariant state,
the inclusion $(\bigotimes_\Gamma \mathcal{O}_\infty) \rca{\zeta} \Gamma \subset D^\sigma$
admits a conditional expectation (see Exercise 4.1.4 of \cite{BO} for instance).

Denote by $\mathbb{Z}_2 \wr \Gamma:= (\bigoplus_\Gamma \mathbb{Z}_2) \rtimes \Gamma$
the wreath product group: the semidirect product of $\bigoplus_\Gamma \mathbb{Z}_2$
by the left shift $\Gamma$-action.
To study properties of $D^\sigma$, we next construct an embedding
\[\rg(\mathbb{Z}_2 \wr \Gamma) \rightarrow \left( \bigotimes_\Gamma \mathcal{O}_\infty \right) \rca{\zeta} \Gamma\] whose image admits a conditional expectation.
Take a unital embedding
\[\iota \colon \rg(\mathbb{Z}_2) \cong \mathbb{C}\oplus \mathbb{C} \rightarrow \mathcal{O}_\infty.\]
Choose minimal projections $q_1, q_2$ in $\iota(\rg(\mathbb{Z}_2))$ with $q_1+q_2=1$.
Take a state $\varphi_i$ on $q_i \mathcal{O}_\infty q_i$
for $i=1, 2$.
Define $E\colon \mathcal{O}_\infty \rightarrow \iota(\rg(\mathbb{Z}_2))$ 
to be $E(x):=\varphi_1(q_1 xq_1)q_1 +\varphi_2(q_2 x q_2)q_2$, $x\in \mathcal{O}_\infty$.
Then $E$ is a conditional expectation.
We equip  $\bigotimes_\Gamma\rg( \mathbb{Z}_2)$ with
the left shift $\Gamma$-action.
Notice that the canonical isomorphism
$ \rg(\bigoplus_\Gamma \mathbb{Z}_2) \cong\bigotimes_\Gamma\rg( \mathbb{Z}_2)$
preserves the left shift $\Gamma$-actions.
Hence it extends to an isomorphism
$\rg(\mathbb{Z}_2 \wr \Gamma)\cong [\bigotimes_\Gamma\rg( \mathbb{Z}_2)]\rc\Gamma $. We identify these two \Cs-algebras via this isomorphism.
Define
$\tilde{\iota} :=\bigotimes_\Gamma \iota \colon \bigotimes_\Gamma\rg( \mathbb{Z}_2)  \rightarrow \bigotimes_\Gamma \mathcal{O}_\infty.$
Then $\tilde{\iota}$ is a $\Gamma$-equivariant unital embedding.
Therefore it extends to an embedding
\[\rg(\mathbb{Z}_2 \wr \Gamma)\rightarrow \left(\bigotimes_\Gamma \mathcal{O}_\infty\right) \rca{\zeta} \Gamma.\]
We show that this inclusion admits a conditional expectation.
Set $\tilde{E}:= \bigotimes_\Gamma E$.
Then $\tilde{E}$ is a $\Gamma$-equivariant conditional expectation of the inclusion $\tilde{\iota}(\bigotimes_\Gamma\rg( \mathbb{Z}_2)) \subset \bigotimes_\Gamma \mathcal{O}_\infty$.
By Exercise 4.1.4 in \cite{BO}, the map $\tilde{E}$ extends to the desired conditional expectation.

Now by Corollary 4 of \cite{Oz} and Theorem 12.3.10 of \cite{BO},
$\rg(\mathbb{Z}_2 \wr \Gamma)$, thus $ D^\sigma$ does not have the completely bounded approximation property.
Since $\rg(\mathbb{Z}_2 \wr \Gamma)$ fails the local lifting property by Corollary 3.7.12 in \cite{BO}, so does $D^\sigma$.
Since $D^\sigma$
is simple and purely infinite,
it is isomorphic to a corner of $(pDp)^\varsigma=p D^\sigma p$.
Thus $\varsigma$ possesses the desired properties.
\end{proof}
\begin{Rem}
Actions stated in Theorem \ref{Thm} and Proposition \ref{Prop}
can be arranged to be \emph{centrally free} in the sense of \cite{Suz19}
(see Definition 4.1 and the sentence below it).
To see this, we choose the action $\alpha$ in the proofs
to satisfy the additional condition that the fixed point algebra contains a unital simple (non-trivial) \Cs-subalgebra. (For instance, replace $\alpha$ by its diagonal action with the trivial action on $\mathcal{O}_\infty$.)
Then the central freeness of the resulting actions
follows from that of the Bernoulli shift actions over unital simple \Cs-algebras;
see Example 4.10 of \cite{Suz19}.
\end{Rem}
\begin{Rem}
Izumi's remarkable theorem (\cite{Izu04}, Theorem 4.2) states that any finite group $G$
admits a (unique) action on $\mathcal{O}_2$ which tensorially
absorbs all $G$-actions on simple unital separable nuclear \Cs-algebras (up to conjugacy).
Theorem \ref{Thm} suggests the non-existence of such an action for
countable infinite amenable groups.
In fact, if we have such $\alpha$,
then by Theorem \ref{Thm},
the fixed point algebra $\mathcal{O}_2^\alpha$ admits conditional
expectations onto
unital isomorphs of the fixed point algebras in Theorem \ref{Thm}.
(In particular $\mathcal{O}_2^\alpha$ fails the operator approximation property
and the local lifting property.)
A related problem is discussed in \cite{Oz04} (see Theorem 4 and the sentence above it).
Observe that the proof of Theorem 4 in \cite{Oz04} shows that,
letting $(\Gamma_i)_{i \in I}$ be a family of (countable) groups as in the statement,
there is no separable \Cs-algebra $A$ with the following property:
For any $i \in I$,
there are completely positive maps
$\varphi_{i} \colon \rg(\Gamma_i) \rightarrow A$
and $\psi_{i} \colon A \rightarrow L(\Gamma_i)$
with $\psi_i \circ \varphi_i = \id_{\rg(\Gamma_i)}$. 
\end{Rem}
\begin{Rem}
Theorem \ref{Thm} and Proposition \ref{Prop} extend to second countable totally disconnected locally compact non-compact groups $\Lambda$.
To see this, choose a decreasing sequence $(K_n)_{n=1}^\infty$
of compact open subgroups of $\Lambda$ with trivial intersection.
Then use the tensor shift action over the $\Lambda$-set $I:=\bigsqcup_{n=1}^\infty \Lambda/ K_n$
instead of the plain Bernoulli shift actions in their proofs.
\end{Rem}
\subsection*{Acknowledgements}
The author is grateful to the referee for careful reading and helpful suggestions.
This work was supported by JSPS KAKENHI Early-Career Scientists
(No.~19K14550) and tenure track funds of Nagoya University.

\end{document}